\documentclass{amsart}

\usepackage{amssymb}
\usepackage[all]{xy}

\numberwithin{equation}{section}

\newtheorem{theorem}{Theorem}[section]
\newtheorem{proposition}[theorem]{Proposition}
\newtheorem{lemma}[theorem]{Lemma}

\theoremstyle{definition}
\newtheorem{definition}[theorem]{Definition}

\theoremstyle{remark}
\newtheorem{remark}[theorem]{Remark}

\newcommand{\wt}{\widetilde}
\newcommand{\tbf}{\textbf}

\newcommand{\into}{\hookrightarrow}

\newcommand{\R}{\mathbb{R}}
\newcommand{\C}{\mathbb{C}}

\newcommand{\Z}{\mathbb{Z}}
\newcommand{\N}{\mathbb{N}}

\newcommand{\Supp}{\text{Supp}}
\newcommand{\U}{\mathcal{U}}
\newcommand{\V}{\mathcal{V}}

\newcommand{\B}{\mathbb{B}}
\newcommand{\K}{\mathbb{K}}

\newcommand{\norm}[1]{\left\| #1 \right\|}
\newcommand{\cstar}{\mathfrak{C}^*}
\newcommand{\dstar}{\mathfrak{D}^*}
\newcommand{\qstar}{\mathfrak{Q}^*}

\newcommand{\trunc}{\mathfrak{T}}

\newcommand{\eps}{\varepsilon}

\newcommand{\End}{\text{End}}
\newcommand{\Ad}{\text{Ad}}

\begin{document}

\title{The Mayer-Vietoris Sequence for the Analytic Structure Group}
\author{Paul Siegel}

\begin{abstract}
We develop a Mayer-Vietoris sequence for the {\em analytic structure group} explored by Higson and Roe.  Using explicit formulas for Mayer-Vietoris boundary maps, we give a new proof and generalizations of Roe's partitioned manifold index theorem (\cite{DualToeplitz}).
\end{abstract}
\maketitle

\section{Introduction}

Let $X$ be a proper metric space.  Higson, Roe, and others have studied a map
\begin{equation*}
\mu \colon K_p(X) \to K_p(C^*(X))
\end{equation*}
from the K-homology of $X$ to the K-theory of the coarse C*-algebra of $X$ called the {\em coarse assembly map}.  It fits into a long exact sequence
\begin{equation} \label{ASES}
\to K_{p+1}(C^*(X)) \to S_p(X) \to K_p(X) \to K_p(C^*(X)) \to
\end{equation}
with a group $S_p(X)$ called the {\em analytic structure group} for its relationship with the structure set in algebraic topology (established in \cite{SurgeryAnalysis1}, \cite{SurgeryAnalysis2}, and \cite{SurgeryAnalysis3}).

The goal of this note is to develop a Mayer-Vietoris sequence for the analytic structure group which is compatible with the ordinary Mayer-Vietoris sequence in K-homology and the Mayer-Vietoris sequence for coarse spaces developed in \cite{CoarseMV}.  Thus for an appropriate decomposition $X = Y_1 \cup Y_2$, these Mayer-Vietoris sequences fit together with \eqref{ASES} to form a braid diagram:

\bigskip

{\tiny
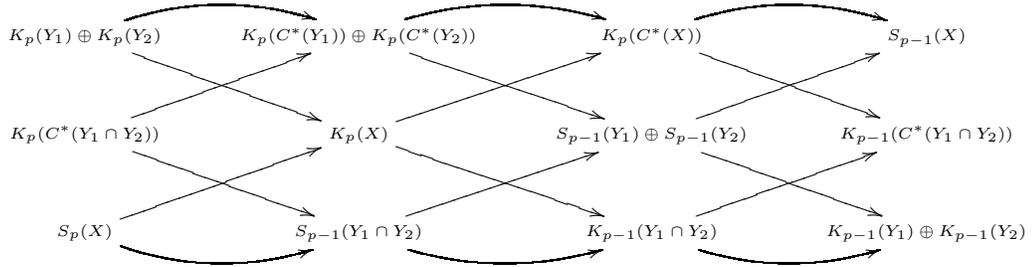
\begin{equation*}
\xymatrix{
K_p(Y_1) \oplus K_p(Y_2) \ar@/^1pc/[r] \ar[dr] & K_p(C^*(Y_1)) \oplus K_p(C^*(Y_2)) \ar@/^1pc/[r] \ar[dr] & K_p(C^*(X)) \ar@/^1pc/[r] \ar[dr] & S_{p-1}(X) \\
K_p(C^*(Y_1 \cap Y_2)) \ar[ur] \ar[dr] & K_p(X) \ar[ur] \ar[dr] & S_{p-1}(Y_1) \oplus S_{p-1}(Y_2) \ar[ur] \ar[dr] & K_{p-1}(C^*(Y_1 \cap Y_2)) \\
S_p(X) \ar[ur] \ar@/_1pc/[r] & S_{p-1}(Y_1 \cap Y_2) \ar[ur] \ar@/_1pc/[r] & K_{p-1}(Y_1 \cap Y_2) \ar[ur] \ar@/_1pc/[r] & K_{p-1}(Y_1) \oplus K_{p-1}(Y_2) 
}
\end{equation*}
}
\bigskip

Exploiting connections between the coarse assembly map and index theory for elliptic operators, we use this braid diagram to give a new proof of a theorem of Roe (\cite{DualToeplitz}).  Roe posed and solved an index problem on a non-compact manifold $M$ which is partitioned by a compact hypersurface $N$, meaning $M$ is the union of two submanifolds whose common boundary is $N$.  Our approach yields new generalizations of Roe's theorem to equivariant indices and to non-compact partitioning hypersurfaces, and as an application we give a simple proof of a classical theorem of Gromov and Lawson (\cite{GromovLawson}) about topological obstructions to the existence of positive scalar curvature metrics.

\section{The Assembly Map and the Analytic Structure Group}

We begin by reviewing the construction of the coarse assembly map and the analytic structure group following the textbook \cite{AnalyticKH}.

\subsection{K-Homology}
The source of the coarse assembly map is the K-homology of the space $X$; we now review a C*-algebraic definition of K-homology.

\begin{definition}
Let $X$ be a locally compact Hausdorff space, let $H$ be a Hilbert space, and let $\rho \colon C_0(X) \to \B(H)$ be a C*-algebra representation.

\begin{itemize}
\item The {\em dual algebra} of $X$ is the C*-algebra $\dstar_{\rho}(X)$ consisting of those operators $T \in \B(H)$ such that $T$ commutes with $\rho(f)$ modulo compact operators for each $f \in C_0(X)$.  Such operators are said to be {\em pseudolocal}

\item The {\em locally compact algebra} is the C*-ideal $\cstar_{\rho}(X)$ in $\dstar(X)$ consisting of those operators $T \in \B(H)$ such that $\rho(f) T$ and $T \rho(f)$ are compact for each $f \in C_0(X)$.  Such operators are said to be {\em locally compact}

\item Let $\qstar_{\rho}(X)$ denote the quotient C*-algebra $\dstar_{\rho}(X)/\cstar_{\rho}(X)$.
\end{itemize}
\end{definition} 

We will often use the notation ``$\sim$'' to denote equality modulo compact operators; thus $T \in \dstar_{\rho}(X)$ if and only if $[T,\rho(f)] \sim 0$ for every $f \in C_0(X)$ and $T \in \cstar_{\rho}(X)$ if and only if $T \rho(f) \sim \rho(f) T \sim 0$ for every $f \in C_0(X)$. 

Note that the dual algebra, the locally compact algebra, and their quotient each depend on the representation $\rho$.  However, their K-theory groups are independent of $\rho$ so long as $C_0(X)$ is separable (equivalently, $X$ is second countable), $H$ is separable, and $\rho$ is ample (meaning $\rho$ is nondegenerate and $\rho(f)$ is a compact operator only if $f = 0$ for $f \in C_0(X)$).  See chapter 5 of \cite{AnalyticKH} for further details.  Because of this we will often abuse notation and write $\dstar(X)$, $\cstar(X)$ and $\qstar(X)$ with the understanding that the representations used to define these C*-algebras are ample and the Hilbert spaces are separable.

\begin{definition}
Let $X$ be a second countable locally compact Hausdorff space.  The {\em $p$th K-homology group} of $X$, $p \in \Z$, is defined to be:
\begin{equation*}
K_p(X) = K_{1-p}(\qstar(X))
\end{equation*}
\end{definition}

\begin{remark}
In fact we shall see that $K_p(\cstar(X)) = 0$, so we could have defined $K_p(X)$ to be $K_{1-p}(\dstar(X))$ as is done in chapter 5 of \cite{AnalyticKH}.
\end{remark}

The K-homology groups form a generalized homology theory (in the sense of the Eilenberg-Steenrod axioms) which is naturally dual to topological K-theory.  In particular they enjoy a certain functoriality property which can be understood at the level of dual algebras as follows.  To begin, let $\phi \colon \U \to Y$ be a continuous proper map from an open subset $\U$ of $X$ to $Y$.  Extend $\phi$ to a map $\phi^+ \colon X^+ \to Y^+$ between one-point compactifications by sending the complement of $\U$ in $X$ to the point at infinity in $Y^+$.  Then $\phi^+$ induces a $*$-homomorphism $C(Y^+) \to C(U^+) \subseteq C(X^+)$ given by $f \mapsto f \circ \phi^+$, and we define $\phi^*$ to be the restriction of this map to $C_0(Y)$.

\begin{definition}
Let $X$ and $Y$ be second countable locally compact Hausdorff spaces, let $\rho_X \colon C_0(X) \to \B(H_X)$ and $\rho_Y \colon C_0(Y) \to \B(H_Y)$ be ample representations on separable Hilbert spaces, and let $\phi \colon \U \to Y$ be a continuous proper map defined on an open subset $\U \subseteq X$.  An isometry $V \colon H_X \to H_Y$ {\em topologically covers} $\phi$ if
\begin{equation*}
V^* \rho_Y(f) V \sim \rho_X(\phi^*(f))
\end{equation*}
for every $f \in C_0(Y)$.
\end{definition}

By Voiculescu's theorem, every map $\phi$ of the sort described in the definition is topologically covered by some isometry.  If $V$ is a topological covering isometry for $\phi$ then the $*$-homomorphism $\Ad_V \colon \B(H_X) \to \B(H_Y)$ maps $\dstar(X)$ into $\dstar(Y)$ and $\cstar(X)$ into $\cstar(Y)$, and the induced maps on K-theory depend only on $\phi$ and not on the chosen covering isometry (\cite{AnalyticKH}, chapter 5).  Thus every continuous proper map $\phi \colon \U \to Y$ on an open subset of $X$ induces a map $\phi_* = (\Ad_V)_* \colon K_p(X) \to K_p(Y)$ on K-homology defined using a topological covering isometry.

\subsection{Coarse K-theory}
The target of the assembly map is the K-theory of the {\em coarse algebra} of the space $X$, so we now review the definition and basic properties of this C*-algebra.  Along the way we will also define the structure algebra.

For much of what follows we will be interested in proper metric spaces, i.e. a metric spaces whose closed and bounded subsets are all compact.  Any complete Riemannian manifold is a proper metric space by the Hopf-Rinow theorem; indeed, this is one our main sources of examples.

\begin{definition}
Let $X$ and $Y$ be proper metric space and let $\rho_X \colon C_0(X) \to \B(H_X)$ and $\rho_Y \colon C_0(Y) \to \B(H_Y)$ be representations.  The {\em support} of a bounded operator $T \colon H_X \to H_Y$ is the set $\Supp(T)$ of all points $(y,x) \in Y \times X$ with the following property: for every open neighborhood $\V \times \U \subseteq Y \times X$ of $(y,x)$ there exist functions $f_1 \in C_0(\U)$ and $f_2 \in C_0(\V)$ such that $\rho_Y(f_2) T \rho_X(f_1) \neq 0$.  An operator in $\B(H_X)$ is said to be {\em controlled} if its support lies within a uniformly bounded neighborhood of the diagonal in $X \times X$.
\end{definition}

The set of all controlled operators forms a $*$-subalgebra, and we will be interested in the intersection of this $*$-subalgebra with the dual algebra and the locally compact algebra.  Most of the results of this paper generalize easily to account for free and proper group actions, so we will include them at the outset (though we will not need to consider equivariant dual algebras or equivariant K-homology).  For simplicity, the reader is invited to ignore the group action during the first reading.    

So let $X$ be a proper metric space equipped with a free and proper action of a countable discrete group $G$ of isometries of $X$.  We will need to consider representations of $C_0(X)$ which are compatible with the $G$-action; the motivating example is the representation of $C_0(X)$ on $L^2(X,\mu)$ by multiplication operators where $\mu$ is a $G$-invariant Borel measure.  More abstractly:

\begin{definition}
Let $H$ be a Hilbert space equipped with a representation 
\begin{equation*}
\rho \colon C_0(X) \to \B(H)
\end{equation*} 
and a unitary representation 
\begin{equation*}
U \colon G \to \B(H)
\end{equation*}
We say that the the triple $(H,U,\rho)$ is a {\em $G$-equivariant $X$-module} or simply a {\em $(X,G)$-module} if $U_\gamma \rho(f) = \rho(\gamma^* f) U_\gamma$ for every $\gamma \in G$, $f \in C_0(X)$.  
\end{definition}

As usual if $H$ carries a unitary representation $U$ of $G$ then we say that an operator $T \in \B(H)$ is $G$-equivariant if $U_\gamma T U_\gamma^* = T$ for every $\gamma \in G$.

\begin{definition}
Let $X$ be a proper metric space equipped with a free and proper action of a countable discrete group of isometries and let $(H_X,U,\rho_X)$ be a $G$-equivariant $X$-module.
\begin{itemize}
\item The {\em structure algebra} of the pair $(X,G)$ is the C*-algebra $D_G^*(X)$ obtained by taking the norm closure of the set of all $G$-equivariant pseudolocal controlled operators.
\item The {\em coarse algebra} of the pair $(X,G)$ is the C*-ideal $C_G^*(X)$ in $D_G^*(X)$ obtained by taking the norm closure of the set of all $G$-equivariant locally compact controlled operators.
\item Let $Q_G^*(X)$ denote the quotient C*-algebra $D_G^*(X)/C_G^*(X)$.
\end{itemize}
\end{definition}

\begin{remark}
It is important that we restrict to $G$-invariant operators before passing to the closure rather than taking the $G$-invariant part of the closure.  Recent examples show that these two procedures do not in general yield the same C*-algebras.
\end{remark}

As before, both $D_G^*(X)$ and $C_G^*(X)$ depend on the representation used to define them, but the ambiguity disappears at the level of K-theory under mild additional hypotheses.  For the coarse algebra it is enough to assume that the Hilbert space is separable and the representation is ample, but for the structure algebra we must assume that the Hilbert space is separable and the representation is the countable direct sum of a fixed ample representation.  Following \cite{AnalyticKH}, we refer to such representations as {\em very ample}.

There is a considerable amount of literature on $C_G^*(X)$ due in part to its relationship with index theory.  Its K-theory groups are functorial for a class of maps between metric spaces called {\em coarse maps} which are compatible with large scale geometry (in the same way that continuous maps are compatible with small scale geometry).  There is a general theory of abstract coarse structures, but for our present purposes it will suffice to review the basic features of the coarse structure of a metric space.

If $S$ is any set and $\alpha_1, \alpha_2 \colon S \to X$ are maps, we say that $\alpha_1$ and $\alpha_2$ are {\em close} if there is a constant $C$ such that $d(\alpha_1(s),\alpha_2(s)) \leq C$ for every $s \in S$.  A map $\phi \colon X \to Y$ between metric spaces is said to be {\em coarse} if $\phi^{-1}(B)$ is a bounded subset of $X$ whenever $B$ is a bounded subset of $Y$ and $\phi \circ \alpha_1$ and $\phi \circ \alpha_2$ are close whenever $\alpha_1$ and $\alpha_2$ are close.  $\phi$ is said to be a {\em coarse equivalence} if there is another coarse map $\psi \colon Y \to X$ such that $\phi \circ \psi$ and $\psi \circ \phi$ are close to the identity maps on $Y$ and $X$, respectively.  For instance, the standard embedding of $\Z$ into $\R$ is a coarse equivalence.

As with K-homology, the functoriality of coarse K-theory can be implemented at the level of the coarse algebra using another notion of covering isometry.

\begin{definition}
Let $X$ and $Y$ be proper metric spaces equipped with representations $\rho_X \colon C_0(X) \to \B(H_X)$ and $\rho_Y \colon C_0(Y) \to \B(H_Y)$.  An isometry $V\colon H_X \to H_Y$ \textit{coarsely covers} a coarse map 
\begin{equation*}
\phi\colon X \to Y
\end{equation*}
if $\pi_1$ and $\phi \circ \pi_2$ are close as maps $\Supp(V) \subseteq Y \times X \to Y$.  If $H_X$ and $H_Y$ carry unitary representations of a group $G$ then we say that a coarse covering isometry $V$ is $G$-equivariant if in addition $V U_g^X = U_g^Y V$.
\end{definition}

According to chapter 6 of \cite{AnalyticKH}, every coarse map $\phi$ admits a coarse covering isometry $V$ so long as $\rho_X$ is nondegenerate and $\rho_Y$ is ample; a straightforward variation on that argument shows that if a discrete group $G$ acts freely and properly on $X$ and $Y$ by isometries and $\phi$ is $G$-invariant then $V$ can be taken to be a $G$-equivariant coarse covering isometry.  Moreover $\Ad_V$ maps $C_G^*(X)$ into $C_G^*(Y)$ and the induced map on K-theory is independent of the $G$-equivariant coarse covering isometry chosen.  Also note that if $V$ coarsely covers $\phi$ then it also coarsely covers any map which is close to $\phi$, so $\phi$ induces a map $\phi_* = (\Ad_V)_* \colon K_p(C_G^*(X)) \to K_p(C_G^*(Y))$ which depends only on the closeness equivalence class of $\phi$.

\subsection{The Analytic Structure Group}

Recall that we defined the structure algebra of a proper $G$-space $X$ to be the C*-algebra $D_G^*(X)$ obtained by taking the norm closure of the space of $G$-invariant pseudolocal controlled operators associated to a representation of $C_0(X)$.

\begin{definition}
Let $X$ be a proper metric space equipped with a free and proper action of a discrete group $G$ by isometries and an ample $(X,G)$-module.  The {\em analytic structure group} of the pair $(X,G)$ is defined to be $S_p(X,G) = K_{1-p}(D_G^*(X))$.  
\end{definition}

While $K_p(X)$ is an invariant of the small scale geometry of $X$ and $K_p(C_G^*(X))$ is an invariant of the large scale geometry of $X$, the analytic structure group depends simultaneously on both large and small scale behavior.  As usual this will be clarified by understanding its functorial properties: we will prove that $S_p(X,G)$ is functorial for {\em uniform maps}, i.e. maps which are both continuous and coarse.  This functoriality is once again implemented using covering isometries.

\begin{definition}
Let $H_X$ and $H_Y$ be Hilbert spaces which carry representations of $C_0(X)$ and $C_0(Y)$, respectively.  An isometry $V \colon H_X \to H_Y$ {\em uniformly covers} a uniform map $\phi \colon X \to Y$ if it both topologically and coarsely covers $\phi$.
\end{definition}

If $X$ and $Y$ carry $G$-actions and $H_X$ and $H_Y$ carry $G$-equivariant representations then one can define $G$-equivariant uniform covering isometries in the same way as in our discussion of $G$-equivariant coarse covering isometries.  $G$-equivariant uniform covering isometries always exist, but the proof of this fact does not appear in the literature and thus we provide it here.  Our strategy is to build "local" covering isometries and patch them together using a partition of unity.

\begin{lemma} \label{GTopIsometry}
Let $X = \U \times G \to \U$ and $Y = \V \times G \to \V$ be trivial $G$-covers of proper metric spaces $\U$ and $\V$ and let $H_X$ and $H_Y$ be ample $(X,G)$- and $(Y,G)$-modules, respectively.  Then any continuous $G$-equivariant map $\phi \colon X \to Y$ is topologically covered by a $G$-equivariant isometry $H_X \to H_Y$.
\end{lemma}
\begin{proof}
Let $H_\U$ denote the closure of $C_0(\U \times \{e\}) H_X$ where $e \in G$ is the identity and define $H_\V$ similarly.  $H_\U$ and $H_\V$ carry ample representations of $C_0(\U)$ and $C_0(\V)$, respectively, and hence there is an isometry $V \colon H_\U \to H_\V$ which topologically covers the restriction of $\phi$ to $U \times \{e\}$.  We have that $H_X \cong \ell^2(G) \otimes H_\U$ and $H_Y \cong \ell^2(G) \otimes H_\V$, so $1 \otimes V \colon H_X \to H_Y$ is an equivariant topological covering isometry for $\phi$.
\end{proof}

To construct $G$-equivariant uniform covering isometries, we break $X$ and $Y$ up into pieces for which the lemma applies.  To make the argument work it is important that the $(Y,G)$-module be very ample, meaning it is the countable direct sum of a fixed ample representation.

\begin{proposition} \label{UniformIsom}
Let $H_X$ be an ample $(X,G)$-module and let $H_Y$ be a very ample $(Y,G)$-module.  Then every equivariant uniform map $\phi \colon X \to Y$ is uniformly covered by an equivariant isometry $V \colon H_X \to H_Y$.
\end{proposition}
\begin{proof}
Assume $H_Y = \bigoplus_\N H$ where $H$ is a Hilbert space carrying a fixed ample representation $\rho_Y \colon C_0(Y) \to \B(H)$.  Let $\pi_X \colon X \to X_G$ and $\pi_Y \colon Y \to Y_G$ be the natural quotient maps for the $G$ action, and choose countable locally finite open covers $\{\U_m\}$ of $X$ and $\{\V_n\}$ of $Y$ with the following properties:
\begin{itemize}
\item Each $\U_m$ is $G$-invariant and evenly covers $\pi_X(\U_m)$, and similarly for $\V_n$.
\item $\{\pi_X(\U_m)\}$ and $\{\pi_Y(\V_n)\}$ have uniformly bounded diameters
\item For every $m$ there exists $n(m)$ such that $\phi(\U_m) \subseteq \V_{n(m)}$
\end{itemize}

Let $H_m^X$ denote the closure of $\rho_X(C_0(\U_m))H_X$ and let $H_n^Y$ denote the closure of $\rho_Y(C_0(\V_n))H$.  Thus $H_m^X$ carries an ample $G$-equivariant representation of $C_0(\U_m)$ and $H_n^Y$ carries an ample $G$-equivariant representation of $C_0(\V_n)$.  Since $\phi$ restricts to a continuous map $\U_m \to \V_{n(m)}$, there is a $G$-equivariant isometry $V_m \colon H_m^X \to H_{n(m)}^Y$ which topologically covers $\phi|_{\U_m}$ by Lemma \ref{GTopIsometry}.

Let $\{h_m\}$ be a $G$-invariant partition of unity subordinate to the open cover $\{\U_m\}$ and define $V \colon H_X \to H_Y = \bigoplus_\N H$ to be the strong limit
\begin{equation*}
V = \bigoplus_m V_m \rho_X(h_m^{1/2})
\end{equation*}

It is clear that $V$ is a $G$-equivariant isometry.  To show that $V$ topologically covers $\phi$, we must show that $\rho_X(g \circ \phi) \sim V^* \rho_Y(g) V$ for every $g \in C_0(Y)$.  We have 
\begin{align*}
V^* \rho_Y(g) V &= \bigoplus_m \rho_X(h_m^{1/2}) V_m^* \rho_Y(g) V_m \rho_X(h_m^{1/2})\\
&\sim \sum_m \rho_X(h_m^{1/2}) \rho_X(g \circ \phi|_{\U_m}) \rho_X(h_m^{1/2})\\
&= g \circ \phi
\end{align*}
since $V_m$ topologically covers $\phi|_{\U_m}$.

Finally, to show that $V$ coarsely covers $\phi$ note that $\Supp(V) = \bigcup_m \Supp(V_m)$ and $\Supp(V_m) \subseteq \V_{n(m)} \times \U_m$.  Since the diameters of the sets $\V_n$ and $\U_m$ are uniformly bounded and $\phi$ maps $\U_m$ into $\V_{n(m)}$, the restrictions of $\pi_1$ and $\phi \circ \pi_2$ to $\Supp(V)$ are close.  This completes the proof.
\end{proof}

If $V$ is a $G$-equivariant covering isometry for a $G$-equivariant uniform map $\phi$ then $\Ad_V$ maps $D_G^*(X)$ into $D_G^*(Y)$ and the induced map on K-theory is independent of the $G$-equivariant covering isometry chosen.  These observations follow immediately by combining the corresponding arguments for K-homology and coarse K-theory.  The conclusion is that every $G$-equivariant uniform map $\phi \colon X \to Y$ induces a map $\phi_* = (\Ad_V)_* \colon S_p(X,G) \to S_p(Y,G)$.

\subsection{The Assembly Map}

We are now ready to construct the (equivariant) coarse assembly map described in the introduction.  As usual let $X$ be a proper metric space on which $G$ acts freely and properly by isometries.  Let $X_G$ denote the quotient space of $X$ by $G$.  There is an isomorphism
\begin{equation} \label{KHControl}
Q_G^*(X) \cong \qstar(X_G)
\end{equation}
and hence the long exact sequence in K-theory associated to the short exact sequence
\begin{equation*}
0 \to C_G^*(X) \to D_G^*(X) \to Q_G^*(X) \to 0
\end{equation*}
takes the form:
\begin{equation} \label{GASES}
\to K_{p+1}(C_G^*(X)) \to S_p(X,G) \to K_p(X_G) \to K_p(C_G^*(X)) \to
\end{equation}
In the case where $X_G$ is a compact manifold and $X$ is its universal cover, this is the {\em analytic surgery exact sequence} investigated by Higson and Roe in \cite{SurgeryAnalysis1}, \cite{SurgeryAnalysis2}, and \cite{SurgeryAnalysis3}.  In any event, the boundary map $K_p(X_G) \to K_p(C_G^*(X))$ is the desired coarse assembly map.

The key is the isomorphism \eqref{KHControl}.  It is proved in \cite{AnalyticKH} and in \cite{PaschkeSheaf} (using sheaf-theoretic techniques).  We will give a brief outline of the argument in \cite{AnalyticKH} so that we can adapt it to the relative setting in the next section.  The main tool is a certain {\em truncation operator} defined as follows:

\begin{definition} \label{truncation}
Let $X$ be a proper metric space and let $\rho \colon C_0(X) \to \B(H_X)$ be a representation.  Let $\{U_n\}$ be a countable locally finite collection of open subsets of $X$ and let $\{h_n\}$ be a subordinate partition of unity.  Given any $T \in \B(H_X)$, define $\trunc(T)$ to be the strong limit of the series $\sum_n \rho(h_n^{1/2}) T \rho(h_n^{1/2})$.  $\mathfrak{T}$ is called the {\em truncation} of $T$ relative to the open cover $\{\U_n\}$.
\end{definition}

Here are the properties of the truncation construction which we will need:

\begin{proposition} \label{TruncFacts}
In the setting of Definition \ref{truncation}, the following hold:
\begin{itemize}
\item $\trunc \colon \B(H_X) \to \B(H_X)$ is a continuous linear map.
\item $\Supp(\trunc(T)) \subseteq \bigcup_n \overline{\U_n \times \U_n}$.
\item If $T$ is pseudolocal then $\trunc(T)$ is pseudolocal and $T - \trunc(T)$ is locally compact.
\end{itemize}
\end{proposition}
\begin{proof}
See Chapter 12 of \cite{AnalyticKH}.
\end{proof}

The isomorphism \ref{KHControl} can be constructed using truncation operators in two steps.  The first step is to pass from equivariant operators associated to $X$ to ordinary operators on $X_G$.  This is more or less straightforward for operators supported in an open subset of $X$ of the form $\U = \U_G \times G$ where $\U_G$ is an open subset of $X_G$, so we cover $X$ by countably many evenly covered open sets of uniformly bounded diameter, lift to an open cover $\{\U_n\}$ of $X$, and truncate a general $G$-equivariant operator $T$ along this open cover using a $G$-invariant partition of unity.  $\trunc(T)$ descends to a pseudolocal controlled operator associated to $X_G$ and it differs from $T$ by a locally compact $G$-equivariant controlled operator by Proposition \ref{TruncFacts}.  This yields an isomorphism $Q_G^*(X) \cong Q^*(X_G)$.

The second step is to pass from a general pseudolocal operator $T$ associated to $X_G$ to a pseudolocal controlled operator.  Once again, the strategy is to truncate $T$ along a uniformly bounded open cover of $X_G$; Proposition \ref{TruncFacts} guarantees that $\trunc(T)$ is pseudolocal and controlled and that it differs from $T$ by a locally compact operator.  This yields an isomorphism $Q^*(X_G) \cong \qstar(X_G)$.  The fact that $\trunc$ is continuous is used in both steps to handle the fact that a general operator in the coarse algebra or the structure algebra is the norm limit of controlled operators but may not itself be controlled.

\section{Mayer-Vietoris Sequences}

In this section we construct Mayer-Vietoris sequences for K-homology, coarse K-theory, and the analytic structure group using an abstract Mayer-Vietoris sequence for C*-algebra K-theory groups.  We also provide formulas for the Mayer-Vietoris boundary maps; these will be used in the next section to prove the partitioned manifold index theorem.  The Mayer-Vietoris sequence for K-homology follows from standard results in classical algebraic topology and the Mayer-Vietoris sequence for coarse K-theory was described in \cite{CoarseMV}, but the author is unaware of any previous discussion of the Mayer-Vietoris sequence for the analytic structure group.

\subsection{The Abstract Mayer-Vietoris Sequence}

Let $A$ be a C*-algebra and let $I_1$ and $I_2$ be C*-ideals in $A$ such that $I_1 + I_2 = A$.  There is a Mayer-Vietoris sequence in K-theory associated to these data which takes the form:
\begin{equation} \label{KTheoryMV}
\to K_{p+1}(A) \to K_p(I_1 \cap I_2) \to K_p(I_1) \oplus K_p(I_2) \to K_p(A) \to
\end{equation}
This is apparently a folklore result among operator algebraists; it appears in \cite{CoarseMV} but it is probably older.  To construct it, consider the C*-algebra $\Omega(A,I_1,I_2)$ defined to be the set of all continuous paths $f \colon [0,1] \to A$ such that $f(0) \in I_1$ and $f(1) \in I_2$.  There is a short exact sequence
\begin{equation} \label{MVShortExact}
0 \to SA \to \Omega(A,I_2,I_2) \to I_1 \oplus I_2 \to 0
\end{equation}
where $SA$ is the suspension of $A$, the first map is inclusion, and the second map is evaluation at $0$ and $1$.  

\begin{lemma}
The $*$-homomorphism $I_1 \cap I_2 \to \Omega(A,I_2,I_2)$ which sends $a \in I_1 \cap I_2$ to the constant path based at $a$ induces an isomorphism in K-theory.
\end{lemma}
\begin{proof}
The space $C[0,1] \otimes I_1 \cap I_2$ of continuous paths in $I_1 \cap I_2$ is an ideal in $\Omega(A,I_1,I_2)$ which is homotopy equivalent to $I_1 \cap I_2$.  Thus it suffices to show that the inclusion $C[0,1] \otimes I_1 \cap I_2$ in $\Omega(A,I_1,I_2)$ induces an isomorphism on K-theory.  The quotient $Q$ of $\Omega(A,I_1,I_2)$ by this ideal is the C*-algebra of paths $g \colon [0,1] \to A/(I_1 \cap I_2)$ such that $g(0) \in I_1/(I_1 \cap I_2)$ and $g(1) \in I_2/(I_1 \cap I_2)$.  Since
\begin{equation*}
A/(I_1 \cap I_2) \cong I_1/(I_1 \cap I_2) \oplus I_2/(I_1 \cap I_2)
\end{equation*}
(this is where $A = I_1 + I_2$ is used), it follows that
\begin{equation*}
Q \cong \left(C_0[0,1) \otimes I_1/(I_1 \cap I_2)\right) \oplus \left(C_0(0,1] \otimes I_2/(I_1 \cap I_2)\right)
\end{equation*}
Thus $Q$ has trivial K-theory and the proof is complete.
\end{proof}

The Mayer Vietoris sequence \eqref{KTheoryMV} can thus be obtained as the long exact sequence in K-theory associated to the short exact sequence \eqref{MVShortExact}

We conclude by providing formulas for the boundary maps in the Mayer-Vietoris sequence.  These formulas use the following device:

\begin{definition}
Let $A$ be a unital C*-algebra and assume $A = I_1 + I_2$ where $I_1$ and $I_2$ are closed C*-ideals in $A$.  A {\em partition of unity} for this decomposition is a pair $(a_1, a_2)$ where $a_j \in I_j$ are positive elements of $A$ satisfying $a_1^2 + a_2^2 = 1$.
\end{definition}

Partitions of unity exist and are unique up to homotopy by a straightforward functional calculus argument.  Let us use a partition of unity to calculate the Mayer-Vietoris boundary map $\partial_{MV} \colon K_1(A) \to K_0(I_1 \cap I_2)$.  This begins with the following lemma.

\begin{lemma}
Let $\tbf{u} \in M_n(A)$ be a unitary representing a class in $K_1(A)$ and define $\tbf{v} = a_1 \tbf{u} + a_2 \tbf{1}$.  We have:
\begin{itemize}
\item $\tbf{v} \sim \tbf{1}$ modulo $M_n(I_1)$
\item $\tbf{v} \sim \tbf{u}$ modulo $M_n(I_2)$
\item $\tbf{v}$ is a unitary modulo $M_n(I_1 \cap I_2)$
\end{itemize}
\end{lemma}
\begin{proof}
Since $a_1 \in I_1$ and $a_2 \in I_2$ are positive elements satisfying ${a_1}^2 + {a_2}^2 = 1$, it follows that $a_2 = \sqrt{1 - {a_1}^2}$ and hence $a_2 \sim 1$ modulo $I_1$.  Thus $a_1 \tbf{u} + a_2 \tbf{1} \sim \tbf{1}$ modulo $M_n(I_1)$, and $\tbf{v} \sim \tbf{u}$ modulo $M_n(I_2)$ follows similarly.  To see that $\tbf{v}$ is a unitary modulo $M_n(I_1 \cap I_2)$, we use the fact that $a_1$ and $a_2$ commute with any element of $A$ modulo $I_1 \cap I_2$ since $a_1 \sim a_2 \sim 1$.  Thus:
\begin{align*}
\tbf{v} \tbf{v}^* &= (a_1 \tbf{u} + a_2 \tbf{1})(\tbf{u}^* a_1 + a_2 \tbf{1})\\
&= ({a_1}^2 + {a_2}^2)\tbf{1} + a_1 \tbf{u} a_2 + a_2 \tbf{u}^* a_1\\
&\sim \tbf{1} + a_1 a_2 \tbf{u} + a_1 a_2 \tbf{u}^*\\
&\sim \tbf{1} \text{ modulo $M_n(I_1 \cap I_2)$}
\end{align*}

Similarly:
\begin{align*}
\tbf{v}^*\tbf{v} &= (\tbf{u}^* a_1 + a_2 \tbf{1})(a_1 \tbf{u} + a_2 \tbf{1})\\
&\sim ({a_1}^2 + {a_2}^2)\tbf{1} + \tbf{u}^* a_1 a_2 + a_2 a_1 \tbf{u}\\
&\sim \tbf{1} + a_1 a_2 \tbf{u}^* + a_1 a_2 \tbf{u}\\
&\sim \tbf{1} \text{ modulo $M_n(I_1 \cap I_2)$}
\end{align*}
\end{proof}

Our formula for $\partial_{MV}$ is as follows:

\begin{proposition}
Let $\tbf{u}$ and $\tbf{v}$ be as in the previous lemma.  Then 
$\partial_{MV} [\tbf{u}] = \partial [\tbf{v}]$ where $[\tbf{v}]$ is the class of $\tbf{v}$ in $K_1(A/(I_1 \cap I_2))$ and $\partial$ is the boundary map for the short exact sequence
\begin{equation*}
0 \to I_1 \cap I_2 \to A \to A/(I_1 \cap I_2) \to 0
\end{equation*}
\end{proposition}

The proof uses an explicit formula for the K-theory boundary map $K_1(A/(I_1 \cap I_2)) \to K_0(I_1 \cap I_2)$ and some elementary homotopies.  Further detail is available in \cite{MyThesis}.

There is a similar formula for the Mayer-Vietoris boundary map in the other degree, but in this case we must assume that $A = I_1 + I_2$ admits a partition of unity $(a_1, a_2)$ such that $a_1$ and $a_2$ are projections; this is required to ensure that $a_1 \tbf{p} + a_2$ is a projection if $\tbf{p}$ is a projection.  Partitions of unity of this sort are available in the examples relevant to this paper.

\begin{proposition}
Assume $A = I_1 + I_2$ admits a partition of unity $(a_1, a_2)$ consisting of projections.  Then the Mayer-Vietoris boundary map $\partial_{MV}: K_0(A) \to K_1(I_1 \cap I_2)$ satisfies $\partial_{MV}[\tbf{p}] = \partial[a_1 \tbf{p} + a_2]$ where $\partial$ is the boundary map associated to the short exact sequence
\begin{equation*}
0 \to I_1 \cap I_2 \to A \to A/I_1 \cap I_2 \to 0
\end{equation*}
\end{proposition}

This formula can be reduced to the formula in the other degree using suspensions and Bott periodicity.  Again, the details are provided in \cite{MyThesis}.

\subsection{The Mayer-Vietoris and Analytic Surgery Diagram}

Let $X$ be a proper metric space and let $Y_1$ and $Y_2$ be subspaces such that $X = Y_1 \cup Y_2$.  In this section we will associate to $Y_1$ and $Y_2$ ideals $\qstar(Y_1 \subseteq X)$ and $\qstar(Y_2 \subseteq X)$ in $\dstar(X)$ and show that $\dstar(X) = \dstar(Y_1 \subseteq X) + \dstar(Y_2 \subseteq X)$.  Combined with some excision results, the abstract Mayer-Vietoris sequence of the previous section will yield a Mayer-Vietoris sequence in K-homology:
\begin{equation*}
\to K_{p+1}(X) \to K_p(Y_1 \cap Y_2) \to K_p(Y_1) \oplus K_p(Y_2) \to K_p(X) \to
\end{equation*}
We will also carry out similar constructions for the coarse K-theory group and the analytic structure group of $X$ and conclude by showing that the three Mayer-Vietoris sequences can be interwoven with the analytic surgery exact sequence to form the braid diagram which appears in the introduction.

To begin, fix a representation $\rho$ of $C_0(X)$ on a separable Hilbert space $H$ and let $Y \subseteq X$ be a subspace.  Say that an operator $T \in \B(H)$ is {\em locally compact for $X-Y$} if $\rho(f)T \sim T\rho(f) \sim 0$ for every $f \in C_0(X-Y)$, and say that $T$ is {\em supported near $Y$} if $\Supp(T) \subseteq B_R(Y) \times B_R(Y)$ for some $R$-neighborhood $B_R(Y)$.

\begin{definition}
Let $X$ be a proper metric space equipped with a free, proper, and isometric $G$-action and let $Y$ be a $G$-invariant subspace.  
\begin{itemize}
\item Let $\dstar(Y \subseteq X)$ denote the C*-ideal in $\dstar(X)$ consisting of pseudolocal operators which are locally compact for $X-Y$.  Let $\qstar(Y \subseteq X)$ denote the ideal $\dstar(Y \subseteq X)/\cstar(X)$ in $\qstar(X)$.
\item Let $C_G^*(Y \subseteq X)$ denote the C*-ideal in $C_G^*(X)$ obtained by taking the norm closure of the set of all $G$-invariant locally compact controlled operators which are supported near $Y$.
\item Let $D_G^*(Y \subseteq X)$ denote the C*-ideal in $D_G^*(X)$ obtained by taking the norm closure of the set of all $G$-invariant pseudolocal controlled operators which are locally compact for $X-Y$ and supported near $Y$. 
\end{itemize}
\end{definition}

\begin{remark}
As before, we do not need to work equivariantly with the dual algebra.
\end{remark}

Defining $Q_G^*(Y \subseteq X) = D_G^*(Y \subseteq X) / C_G^*(Y \subseteq X)$, there is still an isomorphism
\begin{equation*}
Q_G^*(Y \subseteq X) \cong \qstar(Y_G \subseteq X_G)
\end{equation*}
The proof proceeds exactly as before, only now we build the required truncation operator using a collection of open sets $\{\U_n\}$ in $X$ such that $Y \subseteq \bigcup_n \U_n \subseteq B_R(Y)$ for some $R > 0$.  This ensures that $\trunc(T)$ is supported near $Y$ for any operator $T$.  

An crucial fact about these ideals is that their K-theory depends only on $Y$.

\begin{proposition}
Let $X$ and $Y$ be as in the previous definition and fix an ample representation of $C_0(X)$ on a separable Hilbert space.  Then the inclusion $i \colon Y \into X$ is continuous, coarse, and $G$-invariant, and we have:
\begin{itemize}
\item If $Y$ is closed then $i$ induces an isomorphism 
\begin{equation*}
K_p(\qstar(Y)) \cong K_p(\qstar(Y \subseteq X))
\end{equation*}
\item $i$ induces an isomorphism 
\begin{equation*}
K_p(C_G*(Y)) \cong K_p(C_G^*(Y \subseteq X))
\end{equation*}
\item If $Y$ is closed and the representation is very ample then $i$ induces an isomorphism 
\begin{equation*}
K_p(D_G^*(Y)) \cong K_p(D_G^*(Y \subseteq X))
\end{equation*}
\end{itemize}
\end{proposition}
\begin{proof}
For K-homology this is proved in Chapter 5 of \cite{AnalyticKH} and for coarse K-theory this is proved in \cite{CoarseMV}.  For the structure group we proceed as follows.  $i$ is an equivariant uniform map, so it is uniformly covered by an equivariant isometry $V \colon H_Y \to H_X$.  $V$ is in particular an equivariant coarse covering isometry, and following the proof of Proposition \ref{UniformIsom} it is the lift of a topological covering isometry $V_G$ for the inclusion $Y_G \into X_G$.  Thus there is a commutative diagram:
\begin{equation*}
\xymatrix{ 0 \ar[r] & C_G^*(Y) \ar[r] \ar[ddd]^{\Ad(V)} & D_G^*(Y) \ar[r] \ar[ddd]^{\Ad(V)} & D_G^*(Y)/C_G^*(Y) \ar[r] \ar[d]^{\cong} & 0 \\
					          &                                 &                                 & \dstar(Y_G)/\cstar(Y_G) \ar[d]^{\Ad(V_G)} & \\
					          &                                 &                                 & \dstar(Y_G \subseteq X_G)/\cstar(Y_G \subseteq X_G) \ar[d]^{\cong} & \\
					 0 \ar[r] & C_G^*(Y \subseteq X) \ar[r] & D_G^*(Y \subseteq X) \ar[r] & D_G^*(Y \subseteq X)/C_G^*(Y \subseteq x) \ar[r] & 0 }
\end{equation*}
This gives rise to a commutative diagram in K-theory:
{\tiny
\begin{equation*}
\xymatrix{ K_{p+1}(Y_G) \ar[r] \ar[d] & K_{p+1}(C_G^*(Y)) \ar[r] \ar[d] & K_{p+1}(D_G^*(Y)) \ar[r] \ar[d] & K_p(Y_G) \ar[r] \ar[d] & K_p(C_G^*(Y)) \ar[d]\\
					 K_{p+1}(Y_G \subseteq X_G) \ar[r] & K_{p+1}(C_G^*(Y \subseteq X)) \ar[r] & K_{p+1}(D_G^*(Y \subseteq X)) \ar[r] & K_p(Y_G \subseteq X_G) \ar[r] & K_p(C_G^*(Y \subseteq X)) }
\end{equation*}
}
We have proven that all of the vertical maps except the middle one are isomorphisms, so the middle map is an isomorphism by the five lemma.
\end{proof}

We are now ready to construct the desired Mayer-Vietoris sequences associated to a ($G$-invariant) decomposition $X = Y_1 \cup Y_2$.  Let $P_1$ and $P_2$ denote the bounded Hilbert space operators corresponding to bounded Borel functions $f_1$ and $f_2$ which satisfy $f_1 + f_2 = 1$, $f_1|_{X-Y_1} = 0$, and $f_2|_{X-Y_2} = 0$ (such functions exist since $X-Y_1$ and $X-Y_2$ are disjoint open sets).  $P_1$ is a $G$-invariant locally compact controlled operator which is supported near $Y_1$, so in particular $P_1$ is in $\dstar(Y_1 \subseteq X)$, $C_G^*(Y_1 \subseteq X)$, and $D_G^*(Y_1 \subseteq X)$ (and similarly for $P_2$).  Moreover $P_1 + P_2 = 1$, so we have shown that:
\begin{itemize}
\item $\qstar(X) = \qstar(Y_1 \subseteq X) + \qstar(Y_2 \subseteq X)$
\item $C_G^*(X) = C_G^*(Y_1 \subseteq X) + C_G^*(Y_2 \subseteq X)$
\item $D_G^*(X) = D_G^*(Y_1 \subseteq X) + D_G^*(Y_2 \subseteq X)$
\end{itemize}

This is the input required to form the abstract Mayer-Vietoris sequence \eqref{KTheoryMV}, but to obtain a Mayer-Vietoris sequence which depends only on the geometry of the decomposition $X = Y_1 \cup Y_2$ we must check that the intersection of the ideals corresponding to $Y_1$ and $Y_2$ agrees with the ideal corresponding to $Y_1 \cap Y_2$.  This requires certain excision conditions which must be formulated seperately for the dual algebra, the coarse algebra, and the structure algebra.

For the coarse algebra and structure algebra the key condition is that $Y_1$ and $Y_2$ form an {\em $\omega$-excisive pair}, meaning for every $R > 0$ there exists $S > 0$ such that
\begin{equation*}
B_R(Y_1) \cap B_R(Y_2) \subseteq B_S(Y_1 \cap Y_2)
\end{equation*}
For instance if $X = \R$ with the standard metric, $Y_1 = (-\infty,0]$, and $Y_2 = [0,\infty)$ then $Y_1$ and $Y_2$ form a coarsely excisive pair.  If on the other hand we equip $\R$ with a metric obtained by embedding it into $\R^2$ in such a way that the two rays remain within a bounded distance of each other then the decomposition is not $\omega$-excisive.

\begin{proposition} \label{excision}
Let $X = Y_1 \cup Y_2$ be a $G$-invariant decomposition.
\begin{itemize}
\item If $Y_1$ and $Y_2$ are closed then
\begin{equation*}
\dstar(Y_1 \subseteq X) \cap \dstar(Y_2 \subseteq X) = \dstar(Y_1 \cap Y_2 \subseteq X)
\end{equation*}
\item If $Y_1$ and $Y_2$ form an $\omega$-excisive pair then:
\begin{equation*}
C_G^*(Y_1 \subseteq X) \cap C_G^*(Y_2 \subseteq X) = C_G^*(Y_1 \cap Y_2 \subseteq X)
\end{equation*}
\item If $Y_1$ and $Y_2$ are closed and form an $\omega$-excisive pair then
\begin{equation*}
D_G^*(Y_1 \subseteq X) \cap D_G^*(Y_2 \subseteq X) = D_G^*(Y_1 \cap Y_2 \subseteq X)
\end{equation*}
\end{itemize}
\end{proposition}
\begin{proof}
The proof of the identity for the coarse algebra appears in \cite{CoarseMV}, and the identity for the structure algebra follows by combining the identity for the dual algebra with the identity for the coarse algebra.  So we need only prove the identity for the dual algebra.

The containment ``$\subseteq$'' is trivial since $C_0(X - Y_i) \subseteq C_0(X - Y_1 \cap Y_2)$, so that an operator which is locally compact for $X - Y_1 \cap Y_2$ is automatically locally compact for $X - Y_1$ and $X - Y_2$.  To prove ``$\supseteq$'', begin by observing that $X - Y_1 \cap Y_2$ is the disjoint union of $X - Y_1$ and $X - Y_2$.  So given any $f \in C_0(X - Y_1 \cap Y_2)$ we have $f = f_1 + f_2$ where $f_i = f|_{Y_i}$.  For any $T \in \dstar(Y_1 \subseteq X) \cap \dstar(Y_2 \subseteq X)$ we have that $T f_1 \sim f_1 T \sim T f_2 \sim f_2 T \sim 0$, so it follows that $T (f_1 + f_2) \sim (f_1 + f_2) T \sim 0$.
\end{proof}

Now, let $\boldsymbol{\Omega}(C_G^*,X,Y_1,Y_2)$ denote the short exact sequence $\eqref{MVShortExact}$ whose long exact sequence in K-theory is the Mayer-Vietoris sequence for the decomposition $C_G^*(X) = C_G^*(Y_1 \subseteq X) + C_G^*(Y_2 \subseteq X)$, and define $\boldsymbol{\Omega}(D_G^*,X,Y_1,Y_2)$ and $\boldsymbol{\Omega}(Q_G^*,X,Y_1,Y_2)$ similarly.  There is a complex
\begin{equation*}
0 \to \boldsymbol{\Omega}(C_G^*,X,Y_1,Y_2) \to \boldsymbol{\Omega}(D_G^*,X,Y_1,Y_2) \to \boldsymbol{\Omega}(Q_G^*,X,Y_1,Y_2) \to 0
\end{equation*}
Passing to K-theory the columns induce Mayer-Vietoris sequences and the rows induce analytic surgery exact sequences, so by the naturality properties of K-theory we obtain:

{\tiny
\begin{equation*}
\xymatrix{
K_p((Y_1)_G) \oplus K_p((Y_2)_G) \ar@/^1pc/[r] \ar[dr] & K_p(C_G^*(Y_1)) \oplus K_p(C_G^*(Y_2)) \ar@/^1pc/[r] \ar[dr] & K_p(C_G^*(X)) \ar@/^1pc/[r] \ar[dr] & S_{p-1}(X,G) \\
K_p(C_G^*(Y_1 \cap Y_2)) \ar[ur] \ar[dr] & K_p(X_G) \ar[ur] \ar[dr] & S_{p-1}(Y_1,G) \oplus S_{p-1}(Y_2,G) \ar[ur] \ar[dr] & K_{p-1}(C_G^*(Y_1 \cap Y_2)) \\
S_p(X,G) \ar[ur] \ar@/_1pc/[r] & S_{p-1}(Y_1 \cap Y_2, G) \ar[ur] \ar@/_1pc/[r] & K_{p-1}((Y_1 \cap Y_2)_G) \ar[ur] \ar@/_1pc/[r] & K_{p-1}((Y_1)_G) \oplus K_{p-1}((Y_2)_G) 
}
\end{equation*}
}

This is the {\em Mayer-Vietoris and analytic surgery diagram}.

\section{Partitioned Manifolds}

In this section we use the Mayer-Vietoris and analytic surgery diagram to reprove and generalize Roe's index theorem for partitioned manifolds.  Before formulating Roe's theorem, let us take a moment to briefly review the relationship between index theory and the assembly map described in this paper.

\subsection{Index Theory and Assembly}

Let $M$ be a Riemannian manifold and let $D$ be a first order differential operator acting on smooth sections of a vector bundle $S \to M$.  Recall that the {\em symbol} of $D$ is the bundle map $\sigma_D: T^*M \to \End(S)$ obtained by freezing the coefficients of $D$ and passing to the Fourier transform of its top order part.  $D$ is said to be {\em elliptic} if its symbol is invertible away from the zero section, and $D$ has {\em finite propagation speed} if the restriction of $\sigma_D$ to the unit sphere bundle in $T^*M$ is uniformly bounded in norm.  

\begin{definition}
A {\em Dirac-type operator} on a Riemannian manifold $M$ is a first order, symmetric, elliptic differential operator acting on smooth sections of a vector bundle $S \to M$ which has finite propagation speed.
\end{definition}

Other authors sometimes define Dirac-type operators more narrowly than we have here. 

A bundle $S \to M$ is said to be {\em graded} if it comes equipped with a decomposition $S = S^+ \oplus S^-$ into subbundles.  $S$ is $p$-multigraded, $p \geq 1$, if it is graded and it comes equipped with $p$ odd-graded anti-commuting unitary operators $\eps_1, \ldots, \eps_p$ such that $\eps_p^2 = -1$.  Conventionally, a $-1$-multigraded bundle is defined to be an ungraded bundle and a $0$-multigraded bundle is simply a graded bundle.  We say that a differential operator $D$ on $S$ is graded if it sends smooth sections of $S^+$ to smooth sections of $S^-$ and vice-versa, and it is $p$-multigraded if additionally it commutes with all of the multigrading operators.

Any $p$-multigraded Dirac-type operator $D$ on a complete Riemannian manifold $M$ determines a class $[D]$ in the degree $p$ K-homology group $K_p(M)$; see chapter 10 of \cite{AnalyticKH} for the details.  The standard example of a $p$-multigraded Dirac-type operator is the spinor Dirac operator on a $p$-dimensional spin- or spin$^c$-manifold.  If $M$ is a compact manifold then any graded Dirac-type operator $D$ on $M$ is Fredholm, and the Fredholm index of $D$ depends only on its K-homology class.  Indeed, there is a group homomorphism
\begin{equation*}
K_0(M) \to \Z
\end{equation*}
which sends the K-homology class of $D$ to its Fredholm index.  In fact, this map is a special case of the coarse assembly map: since $M$ is compact it is coarsely equivalent to a point, and the coarse algebra of a point is simply the C*-algebra of compact operators.  Thus the coarse assembly map in degree $0$ is simply a map $K_0(M) \to K_0(\K) \cong \Z$, and this turns out to be the index map described above.

The coarse assembly map can be used to construct more general index maps.  Let $M$ once again be a compact Riemannian manifold, let $G$ be its fundamental group, and let $\wt{M}$ be its universal cover.  Since $M$ is compact, $\wt{M}$ is coarsely equivalent to $G$ (regarded as a metric space via any choice of word metric) and one can show that $C_G^*(\wt{M})$ is isomorphic to the {\em reduced group C*-algebra} $C_r^*(G)$ of $G$ (see \cite{AnalyticKH}, chapter 12).  Thus the coarse index map in this case is a map
\begin{equation*}
K_p(M) \to K_p(C_r^*(G))
\end{equation*}
often called the {\em higher index map} or sometimes the {\em equivariant index map}.

\subsection{The Partitioned Manifold Index Theorem}

We are now ready to formulate and prove the partitioned manifold index theorem.

\begin{definition}
Let $M$ be a smooth manifold and let $N$ be a submanifold of codimension $1$.  $M$ is {\em partitioned by $N$} if $M$ is the union of two submanifolds $M^+$ and $M^-$ with common boundary $N$.
\end{definition}

Assume now that $M$ is a complete Riemannian manifold (in particular a proper metric space) and that $N$ is compact.  Let $G$ be a countable discrete group and let $\wt{M} \to M$ be a locally isometric $G$-cover of $M$.  Define a map $M \to \R$ by the formula $x \mapsto \text{dist}(x,N)$; this is a coarse map since $M$ is a length metric space and it lifts to a $G$-invariant coarse map $\wt{M} \to \R \times G$.  Thus it induces a homomorphism $K_p(C_G^*(M)) \to K_p(C_G^*(\R \times G))$.  According to standard results in coarse K-theory (\cite{AnalyticKH}, chapter 6),
\begin{equation*}
K_p(C_G^*([0,\infty) \times G)) = 0
\end{equation*}
and thus the boundary map in the Mayer-Vietoris sequence for coarse K-theory associated to the decomposition $\R \times G = (-\infty, 0] \times G \cup [0,\infty) \times G$ is an isomorphism:
\begin{equation*}
K_p(C_G^*(\R \times G)) \cong K_{p-1}(C_r^*(G))
\end{equation*}

This allows us to construct an index map associated to a partitioned manifold.

\begin{definition}
Let $M$ be a complete Riemannian manifold partitioned by a compact hypersurface $N$ and let $\wt{M}$ be a $G$-cover of $M$ where $G$ is a countable discrete group.  The {\em partitioned index map} is the composition 
\begin{equation*}
\text{Ind}_{M,N}^G \colon K_p(M) \to K_p(C_G^*(\R \times \wt{M})) \cong K_{p-1}(C_r^*(G))
\end{equation*}
\end{definition}

The partitioned manifold index theorem relates the partitioned index of a Dirac-type operator $D$ on $M$ to the ordinary (equivariant) index of its restriction to $N$.  For this to be possible we must make certain assumptions about the local structure of $D$ near $N$:

\begin{definition} \label{PartitionedOperator}
Let $M$ be a smooth manifold partitioned by a hypersurface $N$, let $S_M \to M$ be a smooth $p$-multigraded vector bundle over $M$, and let $S_N \to N$ be a smooth $(p-1)$-multigraded vector bundle over $N$.  Let $D_M$ and $D_N$ be $p$- and $(p-1)$-multigraded differential operators acting on smooth sections of $S_M$ and $S_N$, respectively.  Say that $D_M$ is {\em partitioned by $D_N$} if there is a collaring neighborhood $U \cong (-1,1) \times N$ of $N$ in $M$ with the following properties:
\begin{itemize}
\item $S_M |_U \cong S_{(-1,1)} \hat{\otimes} S_N$ where $S_{(-1,1)}$ is the standard $1$-multigraded complex spinor bundle over $(-1,1)$.
\item $D_M = D_{(-1,1)} \hat{\otimes} 1 + 1 \hat{\otimes} D_N$ where $D_{(-1,1)}$ is the complex spinor Dirac operator on $S_{(-1,1)}$.
\end{itemize}
\end{definition} 

We are now ready to state the main theorem:

\begin{theorem} [The Partitioned Manifold Index Theorem] \label{PMIT}
Let $M$ be a complete Riemannian manifold and let $\wt{M}$ be a locally isometric $G$-cover of $M$ where $G$ is a countable discrete group.  Suppose $M$ is partitioned by a compact hypersurface $N$ and $\wt{M}$ is partitioned by the lift $\wt{N}$ of $N$.  If $D_M$ is a $p$-multigraded Dirac-type operator on $M$ which is partitioned by a $(p-1)$-multigraded Dirac-type operator $D_N$ on $N$ then
\begin{equation*}
\text{Ind}_{M,N}^G[D_M] = \text{Ind}_N^G[D_N]
\end{equation*}
in $K_{p-1}(C_r^*(G))$.
\end{theorem}

Earlier proofs of this theorem, such as in \cite{DualToeplitz} or \cite{CobordismInv}, deal only with the case where $G$ is the trivial group.  Additionally, earlier arguments gave little insight into the where $N$ is non-compact; the proof here uses the compactness hypothesis only to relate the coarse assembly map to the more familiar index map $\text{Ind}_N^G$.  There are a variety of tools for calculating this; if $G$ is trivial then $\text{Ind}_N$ can be calculated explicitly in terms of characteristic classes on $N$ using the Atiyah-Singer index theorem, and for nontrivial $G$ $\text{Ind}_N^G$ can be calculated in many examples using representation theory and equivariant K-theory.

The strategy of the proof is to fit the relevant index maps into a commutative diagram with the boundary map in the Mayer-Vietoris sequence for K-homology:
\begin{equation*}
\xymatrix{
K_p(M) \ar[dr]^{\text{Ind}_{M,N}^G} \ar[dd]_{\partial_{MV}} \\
& K_{p-1}(C_r^*(G)) \\
K_{p-1}(N) \ar[ur]_{\text{Ind}_N^G}
}
\end{equation*}

The result would then follow if we showed that $\partial_{MV}$ sends the K-homology class of $D_M$ to the K-homology class of $D_N$.  It is here that the technical assumptions on $D_M$ in the statement of Theorem \ref{PMIT} play a crucial role.  This is also where our calculations with Mayer-Vietoris boundary maps enter into the proof.

\begin{lemma} \label{MVDirac}
In the setting of Theorem \ref{PMIT} we have
\begin{equation*}
\partial_{MV}[D_M] = [D_N]
\end{equation*}
where $[D_M]$ is the class in $K_p(M)$ determined by $D_M$, $[D_N]$ is the class in $K_{p-1}(N)$ determined by $D_N$, and
\begin{equation*}
\partial_{MV} \colon K_p(M) \to K_{p-1}(M^+ \cap M^-) = K_{p-1}(N)
\end{equation*}
is the boundary map in the Mayer-Vietoris sequence in K-homology associated to the topologically excisive decomposition $M = M^+ \cup M^-$ determined by the partitioning of $M$.
\end{lemma}
\begin{proof}
Let $\varphi^+$ and $\varphi^-$ denote the multiplication operators by the characteristic functions of $M^+$ and $M^-$, respectively.  Observe that these operators are projections and $(\varphi^+, \varphi^-)$ defines a partition of unity for the decomposition $\dstar(M) = \dstar(M^+ \subseteq M) + \dstar(M^- \subseteq M)$.  The K-homology class $[D_M] \in K_p(M) = K_{1-p}(\dstar(M))$ is represented by an operator $F$ in a matrix algebra over $\dstar(M)$ ($F$ is either a projection or a unitary depending on the multigrading structure of $D_M$), and $\partial_{MV}$ sends $[F]$ to the image of $[\varphi^+ F + \varphi^-] \in K_{1-p}(\dstar(M)/\dstar(N \subseteq M))$ under the boundary map
\begin{equation}\label{PartitionMV}
K_{1-p}(\dstar(M)/\dstar(N \subseteq M)) \to K_{-p}(\dstar(N \subseteq M)) \cong K_{p-1}(N)
\end{equation}

According to the excision theorem in K-homology (\cite{AnalyticKH}, chapter 5), 
\begin{equation*}
K_{1-p}(\dstar(M)/\dstar(N \subseteq M)) \cong K_p(M^- - N) \oplus K_p(M^+ - N)
\end{equation*}
so $\partial_{MV}$ sends $[D_M]$ to $\partial[D_M]$ where
\begin{equation*}
\partial \colon K_p(M^+ - N) \to K_{p-1}(N)
\end{equation*}
is the boundary map in K-homology.  Because of the local structure of $D_M$ near $N$, standard results in analytic K-homology (\cite{AnalyticKH}, chapter 11) imply that $\partial_{MV}[D_M] = [D_N]$ (this is essentially an application of the slogan, "the boundary of Dirac is Dirac"). 
\end{proof}

The proof of the partitioned manifold index theorem now follows immediately:

\begin{proof}[Proof of Theorem \ref{PMIT}]
Using the Mayer-Vietoris and analytic surgery diagram, there is a commuting diagram:
\begin{equation*}
\xymatrix{
K_p(M) \ar[rr] \ar[dd]^{\partial_{MV}} \ar[ddrr]^{\text{Ind}_{M,N}^G} & & K_p(C_G^*(\wt{M} \times G)) \ar[dd]^{\partial_{MV}} \\ & & \\
K_{p-1}(N) \ar[rr]^{\text{Ind}_N^G} & & K_{p-1}(C_r^*(G)) 
}
\end{equation*}
Commutativity of this diagram together with Lemma \ref{MVDirac} completes the proof.
\end{proof}

\subsection{Generalizations and Applications}
The main idea of Theorem \ref{PMIT} is to compute the index of an operator on a non-compact manifold $M$ by localizing the calculation to a compact partitioning hypersurface.  This is possible because the partitioning structure provides a mechanism for relating the uniform geometry of $M$ to the uniform geometry of $\R$.  This in turn suggests that a similar result can be proved for a manifold whose uniform geometry is related to that of $\R^k$.

\begin{definition}
Let $M$ be a smooth manifold and let $N$ be a submanifold of codimension $k$.  A {\em $k$-partitioning map} for the pair $(M,N)$ is a coarse submersion $F \colon M \to \R^k$ such that $N = F^{-1}(0)$.  Say that $M$ is {\em $k$-partitioned} by $N$ if there exists a $k$-partitioning map for $(M,N)$.
\end{definition}

We shall consider operators on $M$ which have a favorable local structure near $N$, just as before:

\begin{definition}
Let $M$ be a smooth manifold and suppose $N$ is a submanifold of $M$ of codimension $k$.  Let $D_M$ be a $p$-multigraded differential operator acting on a smooth $p$-multigraded vector bundle $S_M \to M$, $p \geq k$, and let $D_N$ be a $(p-k)$-multigraded differential operator acting on a smooth $(p-k)$-multigraded vector bundle $S_N \to N$.  Say that $D_M$ is {\em $k$-partitioned by $D_N$} if there is a collaring neighborhood $U \cong (-1,1)^k \times N$ of $N$ in $M$ with the following properties:
\begin{itemize}
\item $S_M |_U \cong S_{(-1,1)^k} \hat{\otimes} S_N$ where $S_{(-1,1)^k}$ is the complex spinor bundle on the cube $(-1,1)^k$
\item $D_M = D_{(-1,1)^k} \hat{\otimes} 1 + 1 \hat{\otimes} D_N$ where $D_{(-1,1)^k}$ is the complex spinor Dirac operator
\end{itemize}
\end{definition}

Suppose $M$ is $k$-partitioned by a compact hypersurface $N$ and $\wt{M}$ is a locally isometric $G$-cover of $M$ where $G$ is a countable discrete group.  Suppose further that the $k$-partitioning map $F \colon M \to \R^k$ lifts to a $k$-partitioning map $\wt{F} \colon \wt{M} \to \R^k$ for the pair $(\wt{M},\wt{N})$.  Then $\wt{F}$ induces a map
\begin{equation*}
\wt{F}_* \colon K_p(C_G^*(\wt{M})) \to K_p(C_G^*(\R^k \times G))
\end{equation*}
By induction on $k$ there is an isomorphism $K_p(C_G^*(\R^k \times G)) \cong K_{p-k}(C_r^*(G))$ given by iterated Mayer-Vietoris boundary maps.  Thus we have a $k$-partitioned index map:
\begin{equation*}
\text{Ind}_{M,N}^G \colon K_p(M) \to K_p(C_G^*(\wt{M})) \cong K_{p-k}(C_r^*(G))
\end{equation*}
analogous to the partitioned index map defined above.  Our main result about $k$-partitioned manifolds computes this index map.

\begin{proposition} \label{kPMIT}
Let $M$ be a complete Riemannian manifold and let $\wt{M}$ be a locally isometric $G$-cover of $M$ where $G$ is a countable discrete group.  Let $N$ be a submanifold of $M$ which lifts to a submanifold $\wt{N}$ of $\wt{M}$, and suppose there is a $k$-partitioning map $F \colon M \to \R^k$ for the pair $(M,N)$ which lifts to a $k$-partitioning map $\wt{F} \colon \wt{M} \to \R^k$ for the pair $(\wt{M},\wt{N})$.  If $D_M$ is a $p$-multigraded Dirac-type operator on $M$, $p \geq k$, which is $k$-partitioned by a $(p-k)$-multigraded Dirac-type operator $D_N$ on $N$ then
\begin{equation*}
\text{Ind}_{M,N}^G[D_M] = \text{Ind}_N^G[D_N]
\end{equation*}
in $K_{p-k}(C_r^*(G))$.
\end{proposition}
\begin{proof}
We use induction on $k$; the base case is simply Theorem \ref{PMIT}, so assume $k \geq 2$.  Since $F$ is a submersion the sets $M^+ = F^{-1}(\R^{k-1} \times \R^{\geq 0})$ and $M^- = F^{-1}(\R^{k-1} \times \R^{\leq 0})$ are submanifolds with boundary which partition $M$, and the partitioning hypersurface $N' = M^+ \cap M^-$ is $(k-1)$-partitioned by $N$.  Moreover $D_M$ is partitioned by the $(k-1)$-multigraded operator $D_{N'} = D_{(-1,1)^{k-1}} \hat{\otimes} 1 + 1 \hat{\otimes} D_N$ and $D_{N'}$ is $(k-1)$-partitioned by $D_N$.  By the induction hypothesis it suffices to show that $\text{Ind}_{M,N}^G[D_M] = \text{Ind}_{N',N}[D_{N'}]$ in $K_{p-k}(C_r^*(G))$.  As in the proof of Theorem \ref{PMIT} there is a commutative diagram 
\begin{equation*}
\xymatrix{
K_p(M) \ar[dr]^{\text{Ind}_{M,N}^G} \ar[dd]_{\partial_{MV}} \\
& K_{p-k}(C_r^*(G)) \\
K_{p-1}(N') \ar[ur]_{\text{Ind}_{N',N}^G}
}
\end{equation*}
so the result follows from Lemma \ref{MVDirac}.
\end{proof}

We conclude by providing an application of Proposition \ref{kPMIT} to the theory of positive scalar curvature invariants.  A consequence of the classical Lichnerowicz formula is that if $M$ is a complete Riemannian spin manifold whose scalar curvature function is bounded below by a positive constant and $D$ is the spinor Dirac operator on $M$ then the K-homology class of $D$ lifts to the analytic structure group (\cite{AnalyticKH}, chapter 12).  According to the analytic surgery exact sequence
\begin{equation*}
\to K_{p+1}(C^*(M)) \to S_p(M) \to K_p(M) \to K_p(C^*(M)) \to
\end{equation*}
this implies that the coarse assembly map sends $[D]$ to $0$ in $K_p(C^*(M))$.  This observation together with Proposition \ref{kPMIT} yield a new proof of a celebrated theorem of Gromov and Lawson (\cite{GromovLawson}).

\begin{theorem} [Gromov and Lawson]
Let $M$ be a compact manifold of dimension $n$ which admits a Riemannian metric of non-positive sectional curvature.  Then $M$ has no metric of positive scalar curvature.
\end{theorem}
\begin{proof}
Let $\wt{M}$ be the universal cover of $M$ equipped with a Riemannian metric of non-positive sectional curvature (lifted from a metric $g$ on $M$).  $\wt{M}$ is simply connected, so according to the Cartan-Hadamard theorem the exponential map $\exp_p \colon T_p M \to M$ is a diffeomorphism for any $p$.  In fact we can say more: the exponential map is expansive in the sense that $d(\exp_p(v_1),\exp_p(v_2)) \geq \norm{v_1 - v_2}$ for any $v_1, v_2 \in T_p M$, so the inverse map $\log \colon \wt{M} \to T_p M$ is a coarse diffeomorphism.

Now, let $g'$ be any Riemannian metric on $M$.  Any two norms on a finite dimensional vector space are equivalent, so there is a constant $C$ such that $\frac{1}{C}\norm{v}' \leq \norm{v} \leq C$ for every $v \in T_p M$.  Integrating, it follows that $\frac{1}{C}d'(x,y) \leq d(x,y) \leq C d'(x,y)$ for every $x, y \in M$, so it follows that the map $\log \colon \wt{M} \to T_p M$ defined using the metric $g$ is still a coarse map with respect to the metric $g'$ (lifted to $\wt{M}$).  Thus the $0$-dimensional submanifold $\{p\}$ of $\wt{M}$ is $n$-partitioned by the coarse diffeomorphism $\log \colon \wt{M} \to T_p M \cong \R^n$.  

Clearly $\wt{M}$ is contractible, so let $S_{\wt{M}}$ denote the trivial spinor bundle $\R_n \times \wt{M}$ and let $D_{\wt{M}}$ denote the spinor Dirac operator on $S_{\wt{M}}$ (defined using the Riemannian metric $g'$).  Consider the trivial bundle $S_p = \C \times \{p\} \to \{p\}$ and let $D_p$ be the zero map on $S_p$; $D_{\wt{M}}$ is $k$-partitioned by $D_p$ because $S_{\wt{M}} \cong \R_n \hat{\otimes} S_p$ and $D_{\wt{M}} = D_{\R^n} \hat{\otimes} 1 + 1 \hat{\otimes} D_p$.  By the index theorem for $k$-partitioned manifolds, we have
\begin{equation*}
\text{Ind}_{\wt{M},\{p\}}(D_{\wt{M}}) = \text{Ind}_{\{p\}}(D_p) = 1
\end{equation*}
in $\Z$.  But $\text{Ind}_{\wt{M},\{p\}}$ factors through the coarse index map
\begin{equation*}
K_p(\wt{M}) \to K_p(C^*(\wt{M}))
\end{equation*}
and the image under this map of $D_{\wt{M}}$ would be $0$ if the scalar curvature function on $M$ associated to $g'$ were positive.  Thus $M$ can have no metric of positive scalar curvature.
\end{proof}

\newpage
\bibliographystyle{plain}
\bibliography{refs}

\begin{thebibliography}{10}

\bibitem{GromovLawson}
Mikhael Gromov and Blaine Lawson.
\newblock Positive scalar curvature and the dirac operator on complete
  riemannian manifolds.
\newblock {\em Inst. Hautes Etudes Sci. Publ. Math.}, (58):83--196, 1984.

\bibitem{CobordismInv}
Nigel Higson.
\newblock A note on the cobordism invariance of the index.
\newblock {\em Topology}, 1991.

\bibitem{AnalyticKH}
Nigel Higson and John Roe.
\newblock {\em Analytic K-Homology}.
\newblock Oxford University Press, 2000.

\bibitem{SurgeryAnalysis1}
Nigel Higson and John Roe.
\newblock Mapping surgery to analysis {I}: Analytic signatures.
\newblock {\em K-Theory}, 2005.

\bibitem{SurgeryAnalysis2}
Nigel Higson and John Roe.
\newblock Mapping surgery to analysis {II}: Geometric signatures.
\newblock {\em K-Theory}, 2005.

\bibitem{SurgeryAnalysis3}
Nigel Higson and John Roe.
\newblock Mapping surgery to analysis {III}: Exact sequences.
\newblock {\em K-Theory}, 2005.

\bibitem{CoarseMV}
John~Roe Nigel~Higson and Guo~Liang Yu.
\newblock A coarse mayer-vietoris principle.
\newblock {\em Math. Proc. Cambridge Philos. Soc.}, 1993.

\bibitem{DualToeplitz}
John Roe.
\newblock Partitioning noncompact manifolds and the dual toeplitz problem.
\newblock In {\em Operator Algebras and Applications, Vol 1}. Cambridge
  University Press, 1988.

\bibitem{PaschkeSheaf}
John Roe and Paul Siegel.
\newblock Sheaf theory and paschke duality.
\newblock {\em To Appear}, 2012.

\bibitem{MyThesis}
Paul Siegel.
\newblock {\em Homological Calculations with Analytic Structure Groups}.
\newblock PhD thesis, Penn State, 2012.

\end{thebibliography}
\end{document}